\documentclass[a4paper,11pt,reqno]{amsart}
\usepackage{amsthm}
\usepackage{amsmath}
\usepackage{enumerate}
\usepackage{amsfonts}
\usepackage{amssymb}
\usepackage{fullpage}
\usepackage{amsmath,amscd}
\usepackage{stmaryrd}
\usepackage{graphicx}
\usepackage{array}
\usepackage{amsmath,amssymb,amsfonts,dsfont}
\usepackage[utf8]{inputenc} 
\usepackage[english]{babel}
\usepackage[T1]{fontenc} 
\usepackage{graphicx}
\usepackage{float}
\usepackage{pdfpages}
\usepackage[a4paper, margin = 3cm, bottom = 3cm]{geometry}
\usepackage{ifpdf}
\usepackage{marginnote}
\usepackage{hyperref}	
\hypersetup{pdfborder=0 0 0, 
	    colorlinks=true,
	    citecolor=black,
	    linkcolor=blue,
	    urlcolor=red,
	    pdfauthor={Guillaume Tahar}
	   }
\newtheorem{thm}{Theorem}[section]

\newtheorem{lem}[thm]{Lemma}
\theoremstyle{definition}
\newtheorem{defn}[thm]{Definition}
\theoremstyle{remark}
\newtheorem{rem}[thm]{Remark}
\theoremstyle{definition}

\theoremstyle{definition}

\theoremstyle{definition}

\numberwithin{equation}{section} 
\title{Geometric triangulations and flips}
\author{Guillaume Tahar}
\address[Guillaume Tahar]{Faculty of Mathematics and Computer Science, Weizmann Institute of Science,
Rehovot, 7610001, Israel}
\email{tahar.guillaume@weizmann.ac.il}
\date{June 28, 2019}
\begin{document}
\begin{abstract}
We prove that for a given flat surface with conical singularities, any pair of geometric triangulations can be connected by a chain of flips.\newline
\end{abstract}
\maketitle
\setcounter{tocdepth}{1}
\tableofcontents

\section{Introduction}

Topological surfaces can be studied by combinatorial means through simplicial decompositions.

\begin{defn} A topological triangulation of a topological surface with marked points and boundary (possibly empty) is a maximal family of topological arcs connecting marked points an such that they do not intersect themselves or each other in their interior. The arcs cuts out the surface into ideal triangles (vertices and edges may be not distinct).
\end{defn}

The set of topological triangulations of a given topological surface has a rich combinatorial structure. Elementary transformations of triangulations are called flips.

\begin{defn} A flip is a transformation of an ideal triangulation that removes an edge that is a diagonal of a quadrilateral and replaces it by the other diagonal.
\end{defn}

In the topological setting, it is proved in \cite{FST, Ha} that for a topological surface with marked points, any pair of ideal triangulations can be joined by a chain of flips. In other words, the flip graph of topological triangulations of a topological surface is connected.\newline
The problem is quite different in flat geometry. If we require edges to be geodesic segments, some topological triangulations may not appear as geometric triangulations of a given flat surface. Similarly, an edge of a flat triangulation cannot be flipped when the associated quadrilateral is nonconvex. We prove that a similar theorem holds in the flat setting as well as in the topological setting. The flip graph of geometric triangulations of a flat surface (which is a subgraph of the flip graph of topological triangulations) is also connected.\newline
In the following, a flat surface is a topological compact surface with an everywhere flat metric outside of a finite set of conical singularities (of arbitrary angle) and boundary formed by a finite (possibly empty) union of geodesic segments connecting conical singularities. The boundary does not need to be connected.\newline
In particular, such a definition includes translation surfaces and more generally $1/k$-translation surfaces, see \cite{Zo, BCGGM}. We can also count marked points as conical singularities of angle $2\pi$.

\begin{defn} For a given flat surface, a geometric triangulation is a topological triangulation whose edges (including the boundary) are geodesic segments and whose vertices are conical singularities (every conical singularity should be a vertex of the triangulation).
\end{defn}

\begin{rem}
There are no bigons in flat surfaces of finite area. Consequently, there is at most one geodesic segment in each isotopy class of topological arcs of the surface punctured at the conical singularities.
\end{rem}

\begin{thm}[Main Theorem]
For a given flat surface, any pair of geodesic triangulations can be connected by a chain of flips.
\end{thm}

\section{Results}

\begin{defn} In a flat surface $X$, we define the \textit{singular locus} $Sing(X)$ of $X$ as the union of the conical singularities and the boundary.
\end{defn}

\begin{lem} In any flat surface, there is at least one geometric triangulation.
\end{lem}

\begin{proof}
We consider a flat surface $X$. If $Sing(X)$ is not connected, there is always at least one topological arc without self-intersection that connects two different components of $Sing(X)$. We take the shortest arc $\alpha$ with these requirements. It is a geodesic segment. Iterating this process provides a system of geodesic segments that join any connected components of $Sing(X)$. Cutting along these geodesic segments provides a new connected flat surface $X^{1}$ whose singular locus is connected. Existence of a geometric triangulation in $X^{1}$ implies existence of a geometric triangulation in $X$. Therefore, we reduced the problem to the case where the singular locus is connected.\newline
Using a similar argument, we can once again reduce the problem to the case of flat surfaces with genus zero. If $Sing(X)$ is connected and the genus of $X$ is nonzero, then $X$ contains a topological arc $\beta$ without self-intersection, whose both ends are the same conical singularity and that does not divide the surface. The geodesic representative of $\beta$ is a chain of geodesic segments that do not intersect each other and that do not divide $X$. Cutting along these geodesic segments provides a new flat surface with smaller genus and whose singular locus is still connected. Iterating this process reduces the existence of a geodesic triangulation in $X$ to the existence of a flat triangulation in some flat surface of genus zero.\newline
Following discrete Gauss-Bonnet theorem, there are no flat surfaces of genus zero without boundary and with a unique conical singularity. Therefore, a flat surface of genus zero whose singular locus is connected is a polygon. We still have to prove the lemma in the case of polygons.\newline
In any polygon there is a corner whose magnitude of angle is smaller than $\pi$. This corner $A$ has two neighbors $B$ and $C$ (in the cyclical ordering of the corners of the polygon). We consider the topological arc $\gamma$ that connects $B$ and $C$ staying close to two boundary geodesic segments. Concerning the geodesic representative of $\gamma$, there are three cases:\newline
- it is a unique geodesic segment of the boundary. The flat surface is a triangle and there is a flat triangulation.\newline
- it is a unique geodesic segment that does not belong to the boundary (it is a diagonal of the polygon). Cutting along this segment provides two flat surfaces with strictly fewer corners.\newline
- it is a chain of geodesic segments. Let $D$ be one of the intermediary conical singularity of the chain. There is a geodesic segment between $A$ and $D$ (a diagonal of the polygon). Cutting along this segment provides two flat surfaces with strictly fewer corners.\newline
We have a systematic process to cuts out a polygon into polygons with fewer corners. Besides, the lemma is trivially true for triangles. Therefore, the lemma holds for any polygon. This ends the proof.
\end{proof}

\begin{defn} For a pair of distinct geodesic segments $(\alpha,\beta)$ of a given flat surface, we define $i(\alpha,\beta)$ as the number of intersection points of $\alpha$ and $\beta$ outside their ends. This number is equal to the topological intersection.
\end{defn}

\begin{lem}
Let $X$ be a flat surface with conical singularities that admits two geodesic triangulations $S$ and $T$ such that:\newline
(i) for any pair $(\alpha,\beta) \in S \times T$, we have $i(\alpha,\beta)\geq1$.\newline
(ii) there is an edge $\alpha' \in S$ such that we have $i(\alpha',\beta)=1$ for any $\beta \in T$.\newline
Then $X$ is a convex quadrilateral.
\end{lem}

\begin{proof}
We first choose an orientation of $\alpha'$. Since $\alpha'$ intersects each internal edge of $T$ once, the edges of $T$ can be linearly ordered according to the position of their intersection with $\alpha'$. Among triangles of $T$, we distinguish the two terminal triangles the ends of $\alpha'$ belong to. If the two terminal triangles coincide, then the two ends of $\alpha'$ belong to the same of corner of the terminal triangle because otherwise $\alpha'$ would intersect itself. Thus, $\alpha'$ crosses at least two times the edge opposite to this corner. Therefore, the two terminal triangles are distinct.\newline
In each of the two terminal triangles, $\alpha'$ intersects once the edge opposite to the end corner. The two other edges of the triangle belong to the boundary of the surface. Indeed, if they were internal edges, then $\alpha'$ would cross them at least one time and $\alpha'$ would either intersect itself or cross an internal edge more than once.\newline
Concerning triangles other than terminal, since the number of crossings of $\alpha'$ with the boundary of each of these triangles is even, exactly one edge of these triangles belongs to the boundary of the surface. Therefore, there is also a linear order for the triangles of the triangulation. The geometry of the flat surface is very restricted. It is a polygon, that is a topological disk with a boundary formed by a chain of geodesic segments. Since every internal edge of $S$ crosses each internal edge of $T$ once, its ends belong to the same corners of the same terminal triangles. They passes through each internal edge of $T$ in the same order.\newline
Consequently, any other edge of $S$ is isotopic to $\alpha'$. There are no bigons in flat surfaces so there is unique internal edge in S. Therefore, $X$ is a quadrilateral. Since there is another triangulation T whose unique internal edge crosses $\alpha'$, $X$ is a convex quadrilateral.
\end{proof}

\begin{proof}[Proof of the main theorem]
The number of triangles in a flat triangulation of a given flat surface is determined by the sum of the total angles of the conical singularities. In addition, flat triangles have a total angle of $\pi$. Therefore, each geometric triangulation of a given flat surface has the same number of triangles. As a consequence, they have the same number of internal edges (edges that do not belong to the boundary).
We suppose there exists a flat surface $X$ with a minimal number of internal edges and such that there are two flat triangulations $S$ and $T$ of $X$ that cannot be joined by a chain of flips.\newline
Triangulations $S$ and $T$ cannot have a common internal edge because we could cut along this edge and get a flat surface with two additional boundary edges and one less internal edge. The two induced geometric triangulations of this new surface cannot be joined by a chain of flips. Thus, the minimality hypothesis is negated.
Similarly, we prove that an edge $x$ of $S$ and an edge $y$ of $T$ always have a nontrivial intersection. Indeed, if there were such a pair of edges, we could complete $x$ and $y$ to get a new geometric triangulation $Z$. Indeed, any flat surface with conical singularities admits a geometric triangulation, see Lemma 2.2. Using the previous argument of minimality, $S$ and $Z$ would have a common edge and could be joined by a chain of flips. It is the same for $T$ and $Z$. Therefore, $S$ and $T$ could be joined by a chain of flips.\newline

For any pair of geodesic segments $(\alpha,\beta)$ of $X$ with nontrivial intersection, we can define a pair of topological arcs $(\gamma,\delta)$ that is a desingularization of the pair $(\alpha,\beta)$. These topological arcs (without self-intersection) follow the geodesic segments outside a neighborhood of the intersection points of $\alpha$ and $\beta$. Near these intersection points, they are drawn in such a way that $\gamma$ and $\delta$ do not intersect each other. In addition, they do not intersect either $\alpha$ or $\beta$. In fact, there are two ways to provide such a desingularization depending on which endpoints of $\alpha$ and $\beta$ are connected by desingularized arcs.\newline
Homotopy classes of topological arcs $\gamma$ and $\delta$ have geodesic representatives that are a chain of geodesic segments. These geodesic representatives minimize self-intersection number and minimize geometric intersection with $\alpha$ and $\beta$ in their homotopy class.\newline
We are going to prove that in our minimal hypothetic counter-example, the singular locus is connected. We suppose that it is not. In each of the two triangulations $S$ and $T$ that cannot be joined by a chain of flips, we consider internal edges $\alpha \in S$ and $\beta \in T$ that join two connected components of the singular locus. Indeed, in any triangulation there is always enough internal edges to relate the connected components of the singular locus. The two edges $\alpha$ and $\beta$ intersect each other. Desingularization provides two chains of geodesic segments that relates several connected components of the singular locus. Any geodesic segment of these chains does not intersect $\alpha$ and $\beta$. Therefore this segment $u$ belongs to the boundary of $X$. Otherwise, we could provide a triangulation $S'$ that contains both $u$ and $\alpha$. In the same way, there would be another triangulation $T'$ with $u$ and $\beta$. Cutting along $u$ provides a simpler flat surface (with fewer internal edges) where any two geometric triangulations are joined by a chain of flips. Therefore, $S'$ and $T'$ are joined by a chain of flips. Since $S$ and $S'$ (similarly for $T$ and $T'$) share an internal edge, they are joined by a chain of flips (using the same argument). Consequently, any such segment $u$ belongs to the boundary of the surface. So, the singular locus of $X$ is connected.\newline
Similarly, we prove that our minimal hypothetic counter-example has genus zero. We suppose its genus is nonzero. Then, in any triangulation of $X$, there is at least one internal edge such that cutting along it does not disconnect the surface. In each of the two triangulations $S$ and $T$ that cannot be joined by a chain of flips, we consider internal edges $\alpha \in S$ and $\beta \in T$ that do not disconnect the surface. These two geodesic segments have a nontrivial intersection. After desingularization, we get a topological arc $\gamma$ that does not intersect $\alpha$ nor $\beta$. Arc $\gamma$ does not disconnect the surface so its geodesic representative does not disconnect the surface either. It is true for any geodesic segment of the chain. In the same way as we proved connectedness of the singular locus, from existence of such a geodesic segment $u$ that does not intersect $\alpha$ nor $\beta$ we can construct a chain of flips that join $S$ and $T$. Therefore, our minimal counter-example is a flat surface of genus zero whose singular locus is connected.\newline
A flat surface of genus zero whose singular locus is connected is a polygon, that is a topological disk whose boundary is a cyclic chain of geodesic segments. There should be at least three segments in the boundary because otherwise, angles in the inner corners would be degenerated.\newline
In polygons, two geodesic segments have at most one intersection otherwise they would form a flat bigon. If there were two geometric triangulations that could not be joined by a chain of flips, then they would satisfy the hypothesis of Lemma 2.4 and the polygon would be a convex quadrilateral. This ends the proof.
\end{proof}

\textit{Acknowledgements.} I thank Dylan Thurston for his valuable remarks. The author is supported by a fellowship of Weizmann Institute of Science. This research was supported by the Israel Science Foundation (grant No. 1167/17).\newline

\nopagebreak
\vskip.5cm
\end{document}